\documentclass[a4paper, 12pt]{article}
\usepackage[margin=1in]{geometry} 
\usepackage{amsmath,amssymb,amsthm}
\usepackage{amsmath}
\usepackage{amsfonts}
\usepackage{hyperref}
\usepackage{mathrsfs}
\usepackage{mathtools}
\usepackage{graphicx}
\usepackage{tikz-cd}
\usepackage{parskip}
\usepackage{xcolor}
\usepackage[capitalize]{cleveref}
\usepackage[overload]{empheq}

\newcommand{\RR}{\mathbb{R}}
\newcommand{\ZZ}{\mathbb{Z}}
\newcommand{\NN}{\mathbb{N}}
\newcommand{\QQ}{\mathbb{Q}}

\newcommand{\CC}{\mathbb{C}}
\newcommand{\DD}{\mathbb{D}}

\newcommand{\condRR}{\underline{\RR}}
\newcommand{\condZZ}{\underline{\ZZ}}

\newcommand{\QQdisc}{\QQ^\text{disc}}
\newcommand{\condQQdisc}{\underline{\QQdisc}}
\newcommand{\condRRdisc}{\underline{\RR^\text{disc}}}
\newcommand{\condext}{\underline{\ext}}
\newcommand{\RRdisc}{\RR^\text{disc}}
\newcommand{\condRRZZ}{\underline{\RR/\ZZ}}
\newcommand{\condsethom}{\mathscr{H}\kern -.5pt om}
\newcommand{\homcondset}{\Hom_{\operatorname{Cond(Set)}}}
\newcommand{\homcondab}{\Hom_{\condab}}
\newcommand{\condset}{\operatorname{Cond(Set)}}
\newcommand{\homab}{\Hom_{\ab}}
\newcommand{\ZZTr}{\ZZ((T))_r}
\newcommand{\condZZTr}{\underline{\ZZTr}}
\newcommand{\ZZTrc}{\ZZ ((T))_{r, \leq c}}
\newcommand{\condZZTrc}{\underline{\ZZTrc}}

\newcommand{\Hicech}{H^i_{\text{\v{C}ech}}}

\newcommand{\mc}[1]{\mathcal{#1}}

\DeclareMathOperator{\Hom}{Hom}
\DeclareMathOperator{\ab}{Ab}

\DeclareMathOperator{\condab}{Cond(Ab)}
\newcommand{\rhom}{R\Hom}
\newcommand{\condrhom}{R\underline{\Hom}}
\newcommand{\condhom}{\underline{\Hom}}

\DeclareMathOperator{\tot}{Tot}
\DeclareMathOperator{\ext}{Ext}
\DeclareMathOperator{\im}{Im}

\DeclareMathOperator{\colim}{colim}

\DeclareMathOperator{\holim}{holim}

\renewcommand{\epsilon}{\varepsilon}

\DeclarePairedDelimiter{\norm}{\lVert}{\rVert}

\usepackage{blindtext}

\usepackage{thmtools,thm-restate}

\newtheorem{theorem}{Theorem}
\newtheorem{lemma}[theorem]{Lemma}
\newtheorem{prop}[theorem]{Proposition}
\newtheorem{cor}[theorem]{Corollary}

\theoremstyle{definition}
\newtheorem{defi}[theorem]{Definition}
\newtheorem{rem}[theorem]{Remark}
\newtheorem{example}[theorem]{Example}

\numberwithin{equation}{section}
\numberwithin{lemma}{section}
\numberwithin{theorem}{section}
\numberwithin{cor}{section}
\numberwithin{defi}{section}
\numberwithin{rem}{section}
\numberwithin{example}{section}
\numberwithin{prop}{section}

\begin{document}

\begin{titlepage}
   \begin{center}
       \vspace*{1cm}
       
       \huge
       \textbf{Condensed Mathematics} \\
		 \LARGE
		 \vspace{0.7cm}
		 The internal Hom of condensed sets and condensed abelian groups and a prismatic construction of the real numbers
		 \huge
            
	  	 \vspace{4cm}
       Rodrigo Marlasca Aparicio
       \vspace{2cm}
            
       Part B Extended Essay \\
		 \vspace{1cm}
		 Merton College, University of Oxford \\
		 \vspace{1cm}
		 MMath Mathematics \\
       \vspace{1cm}
       Trinity Term 2021
            
       \vspace{0.8cm}

   \end{center}
\end{titlepage}
	
{
  \hypersetup{linkcolor=black} 

  \renewcommand{\baselinestretch}{1.75}\normalsize 
  \tableofcontents
}

\section{Introduction}

Condensed mathematics are a very recent area of study, first proposed by Dustin Clausen and Peter Scholze. They provide a generalization of topological spaces, groups, rings and modules that is arguably better behaved than the usual categories.

One of the most common motivating examples is the map $\RRdisc \to \RR$ in the category of topological abelian groups. This is clearly an epimorphism and a monomorphism, so if we had an abelian category we would expect an isomorphism of topological groups. However, this is not the case due to the different topologies, making many constructions difficult in this category. If we consider this as a map of condensed abelian groups, it can be shown that it is not an epimorphism, so the cokernel is not trivial.

Having an abelian category of topological groups has many advantages. Mainly, one can start using many homological methods like extension and torsion groups. For example, while continuous group cohomology does not always give long exact sequences, we can substitute these by extension groups in the abelian category of condensed $\ZZ[G]$-modules, fixing this problem.

This new theory has yielded some surprising results and not only with topological groups. An example is the Whitehead problem: is every abelian group with $\ext^1 (A, \ZZ) = 0$ a free abelian group? This question turns out to be independent of ZFC \cite{whitehead}, but it has been proved that if $\condext^1 (\underline{A}, \condZZ) = 0$ then $A$ is free abelian, where the extension group is enriched over condensed abelian groups \cite{masterclass}.

Very recently, condensed mathematics have also been used to provide a more categorical and abstract approach to functional analysis that has been coined analytic geometry \cite{analytic}.

Other theories have been developed to solve the aforementioned issues. For example, pyknotic objects \cite{pyknotic} use the same core ideas but set theoretic difficulties are solved differently. Finally, a different approach is given by bornological modules \cite{bornology}, which form a quasi-abelian category but avoid some unintuitive problems of condensed sets like non-trivial spaces with the empty set as its underlying set. However, this essay will solely focus on condensed objects and some of their applications.

\begin{itemize}
	\item \cref{section-condobjects} provides an introduction to the study of condensed objects, paying particular attention to condensed abelian groups and their resolutions and cohomology. 
	\item In \cref{section-enrichedhom} we talk about closed symmetric monoidal structures in various condensed categories. We define the internal Hom for condensed sets and provide an alternative proof of \cite{condensed}, Proposition 4.2.
	\item \cref{section-derivedmapping} shows a few important examples of the enriched Hom in the derived category. 
	\item \cref{section-prismatic} studies an intermediate step of a theorem of Clausen and Scholze, constructing the real numbers from discrete spaces.
\end{itemize}

\section{Condensed objects}
\label{section-condobjects}

\subsection{Preliminaries}

Sites give a more generalised version of a topological space. Instead of open sets, we will have sieves.

\begin{defi}
	Let $\mc{C}$ be a category and let $c$ be an object of $\mc{C}$. A \emph{sieve} $S: \mc{C}^{op} \to \text{Set}$ is a subfunctor of $\text{Hom} (-, c)$.	

	Given a morphism $f: c' \to c$ we can define the \emph{pullback} of $S$ by $f$ as
	\begin{align*}
		f^*S (c'') = \{ g: c'' \to c' | fg \in S(c'')\}
	\end{align*}
	That is, arrows that composed with $f$ give a map in the sieve.
\end{defi}

\begin{defi}
	A \emph{Grothendieck topology} on a category $\mc{C}$ is a collection, for each object $c$ of $\mc{C}$, of sieves called the \emph{covering sieves} of $c$. They must satisfy the following axioms.
	\begin{enumerate}
		\item	If $S$ is a covering sieve on $X$ and $f: Y \to X$ is a morphism, then the pullback $f^* S$ is a convering sieve on $Y$.
		\item	If $S$ is a covering sieve on $X$, $T$ is a sieve on $X$ and for any $f \in S$ the sieve $f^*T$ is a covering sieve, then $T$ is a covering sieve.
		\item $\Hom(-, X)$ is a covering sieve for any $X$ in $\mc{C}$.
	\end{enumerate}
	A \emph{site} is a category $\mc{C}$ together with a Grothendieck topology.
\end{defi}

\begin{rem}
	We make the distinction between a topology and a pretopology. A \emph{covering family} is a collection of maps with a common codomain.

	A \emph{Grothendieck pretopology} is a collection of covering families such that
	\begin{enumerate}
		\item For all objects $X$ of $\mc{C}$, all maps $Y \to X$ and all covering families $\{X_i \to X\}$, all the relevant fiber products exist and the collecion $\{X_i \times _X Y \to Y\}$ is a covering family.
		\item If $\{X_i \to X\}$ is a covering family and for every $i$ the collection $\{X_{ij} \to X_i\}$ is also a covering family, the composition $\{X_{ij} \to X_i \to X\}$ is a covering family.
		\item If $f$ is an isomorphism, $\{f\}$ is a covering family.
	\end{enumerate}
	Every pretopology can be extended to a topology by considering all sieves containing a covering family \cite[Remark 1.3.1]{topos}.
\end{rem}

We may also define a sheaf on a site. Remember that if we consider the category of posets of a topological space $X$, if $U, V \subset X$ are open then $U \cap V = U \times _X V$, which gives rise to the following natural definition.

\begin{defi}
	An $X$-valued \emph{sheaf} on a site $\mc{C}$ is a functor $F: \mc{C}^{op} \to X$ such that for any $X \in \mc{C}$ and any cover $\{X_i \to X\}$ the following diagram is an equaliser
	\begin{align*}
		F(X) \xhookrightarrow{} \prod_i F(X_i) \rightrightarrows \prod_{i,j} F(X_i \times _X X_j) 
	\end{align*}
	where the stacked arrows correspond to the induced maps from the fiber product to $X_i$ and $X_j$.
\end{defi}

\begin{rem}
	The inclusion $F(X) \xhookrightarrow{} \prod_i F(X_i)$ is the local character of the sheaf, since the value of $F(X)$ is determined by the value at a cover. On the other hand, the equaliser corresponds to the gluing axiom, because sections coinciding at the fiber product assemble together to form the section at $X$.
\end{rem}

We now define a particularly important site for the study of condensed objects.

\begin{defi}
	A \emph{profinite set} is a Hausdorff, compact and totally disconnected space. 

	The \emph{pro-étale site} $*_{\text{proét}}$ of a point is the category of profinite sets $S$ whose covering families are finite sets of jointly surjective maps.
\end{defi}

\begin{rem}
	The above terminology comes from the pro-étale site of a scheme $X$, since the site of profinite sets corresponds to the pro-étale site of $X$ when it is just a single point \cite{etale}.

	To check this is a pretopology, observe that a limit of profinite spaces is profinite \cite[\href{https://stacks.math.columbia.edu/tag/0ET8}{Tag 0ET8}]{stacks-project}, so all relevant fiber products exist.

	For maps of topological spaces
	\[
		\begin{tikzcd}
			A \arrow{r}{f} & C & \arrow{l}[swap]{g} B
		\end{tikzcd}
	\]
	the fiber product $A \times _C B$ is given by $\{(a,b) | f(a) = g(b)\} \subseteq A \times B$, and we have obvious projection maps to $A$ and $B$ \cite{calcut2012topological}.

	Then, for a covering family $\{X_i \to X\}$ (so they are finite and jointly surjective) and a map $f: Y \to X$, for every $y \in Y$, there is some $x$ in some $X_i$ such that their images onto $X$ coincide. The fiber products then surject onto $Y$ and pullback gives a new covering family.

	The composition of two covering families is clearly surjective and any isomorphism is surjective, so it forms a covering family. This shows all pretopology axioms are satisfied.
\end{rem}

One would want to naively define condensed sets as a sheaf on $*_{\text{proét}}$. Evaluation at a profinite set can be thought of as the maps from this profinite set into our condensed set. In this way, condensed sets capture not only the topology of the space but also how it interacts with other spaces.

However, the category of profinite sets is large and this would present many set theoretic issues, like the category not being locally small. We use extremally disconnected sets to find an equivalent definition.

\begin{defi}
	An \emph{extremally disconnected set} is a projective object in the category of compact Hausdorff spaces.
\end{defi}

\begin{rem}
	\cite{Gleason} shows that the category of compact Hausdorff spaces has enough projectives, so any compact Hausdorff $X$ admits a surjection $S \twoheadrightarrow{} X$ where $S$ is extremally disconnected.

	This implies that a sheaf on profinite sets is determined by the values at extremally disconnected sets, and it can be shown that the category of sheaves on profinite sets is equivalent to the category of functors on extremally disconnected sets taking finite disjoint unions to products \cite[Proposition 2.7]{condensed}.

	Extremally disconnected sets are difficult to come by (the most common ones are the Stone-\v{C}ech compactification of a discrete space) but their characterisation as projectives will be enough for this essay.
\end{rem}

\begin{defi}
	A condensed set/group/ring is a functor
	\begin{align*}
		T: \{\text{extremally disconnected sets}\}^{\text{op}} \to \{\text{sets/groups/rings}\}
	\end{align*}
	such that $T(\emptyset) = *$ and for any extremally disconnected $S_1, S_2$,
	\begin{align*}
		T(S_1 \sqcup S_2) = T(S_1) \times T(S_2)
	\end{align*}
	and is the left Kan extension of its restriction to $\kappa$-small extremally disconnected sets, for some uncountable strong limit cardinal $\kappa$.
	\[
		\begin{tikzcd}
			\text{ED}^{\text{op}} \arrow[rd, dashed, "T"] & \\
			\text{ED}_{\kappa}^{\text{op}} \arrow{u} \arrow{r}{T_{\kappa}} & \{\text{sets/groups/rings}\}
		\end{tikzcd}
	\]
\end{defi}

\begin{rem}
	Equivalently, they are also sheaves on profinite sets satisfying a similar Kan extension condition. It will then make sense to talk about sheafification.

	Informally, one can think of the left Kan condition as the sheaf being determined by its values on $\kappa$-small extremally disconnected sets for some $\kappa$. This is not exactly a category of sheaves on a site, but it shares many of its features.

	\cite[Remark 2.13]{condensed} shows that this category is stable under the section-wise formation of limits and filtered colimits, so these will be easily computed.
\end{rem}

\begin{prop}[\cite{condensed}, Proposition 2.15]
	There is a functor from T1 topological spaces/groups/rings to condensed sets/groups/rings sending $T \to \underline{T}$, where $T(S) = C(S,T)$ for extremally disconnected $S$.

	If $A$ is compactly generated as a topological space and $A,B$ are T1 topological spaces/groups/rings then
	\begin{align*}
		\Hom(\underline{A}, \underline{B}) = \Hom(A,B)
	\end{align*}
	where the right hom-set is in the category of topological spaces/groups/rings.
\end{prop}

\begin{rem}
	The T1 condition on the topology is to ensure that the defined sheaf is the left Kan extension of the restriction to $\kappa$-small extremally disconnected sets for some $\kappa$ (\cite{condensed} Proposition 2.15).

	This functor allows us to think about condensed sets in a more natural way. Let $T$ be a condensed set and let $S$ be extremally disconnected. Then,
	\begin{align*}
		\Hom (\underline{S}, T) = T(S)
	\end{align*}
	by the Yoneda lemma. Hence, evaluation at $S$ can be thought of as the set of continuous maps from $S$ to $T$.
\end{rem}

\subsection{Condensed abelian groups}

Condensed abelian groups turn out to be as well-behaved as one could reasonably expect them to be.

\begin{prop}[\cite{condensed}, Theorem 2.2]
	The category of condensed abelian groups is abelian, satisfying the same Grothendieck axioms as the category of groups. That is,
	\begin{enumerate}
		\item (AB3) All colimits exist.
		\item (AB3*) All limits exist.
		\item (AB4) Arbitrary direct sums are exact.
		\item (AB4*) Arbitrary products are exact.
		\item (AB5) Arbitrary filtered colimits are exact.
		\item (AB6) For any index set $J$ and filtered categories $I_j, j \in J$ with functors $i \to M_i$ from $I_j$ to condensed abelian groups,
			\begin{align*}
				\lim_{i_j, \forall j \in J} \prod_{j \in J} M_{ij} \cong \prod_{j \in J} \lim_{i_j \in I_j} M_{ij}
			\end{align*}
	\end{enumerate}
\end{prop}

\begin{rem}
	Since all limits and colimits of abelian groups commute with finite products (=direct sums) we have that the category of functors on extremally disconnected spaces taking finite disjoint unions to products is stable under section-wise limits and colimits. 

	When passing to condensed sets, note that a left Kan extension can be written as a filtered colimit, so it will commute with any colimits. As for limits, \cite[Remark 2.13]{condensed} shows that they also preseve the left Kan condition by choosing a suitable cardinal $\kappa$.

	Limits and colimits being computed section-wise leads to naturally talking about injections or surjections of condensed abelian groups (instead of monomorphisms and epimorphisms). Unlike the usual category of sheaves, both kernels and cokernels can be computed at every section. A map is a monomorphism (resp. epimorphism) if and only if it is injective (resp. surjective) at every section.

	We highlight the contrast with condensed sets, where one can only take limits and filtered colimits section-wise. The problem with a general colimit is that they don't commute with all products in the category of sets so the sheaf condition might not be preserved.
\end{rem}

There is a condensed analogue to free abelian groups.

\begin{prop}[\cite{condensed}, Theorem 2.2]
	\label{prop-projcondab}
	There is a left adjoint functor to the forgetful functor from condensed abelian groups to condensed sets, $T \mapsto \ZZ[T]$. This is the sheafification of $S \mapsto \ZZ[T(S)]$.

	For extremally disconnected $S$, the free condensed abelian group $\ZZ[\underline{S}]$ is projective.
\end{prop}

\begin{rem}
	Using this adjunction, we see that for extremally disconnected $S$ and a condensed abelian group $M$,
	\begin{align*}
		\Hom (\ZZ[\underline{S}], M) \cong \Hom (\underline{S}, M) \cong M(S)	
	\end{align*}
	where the last isomorphism is by the Yoneda lemma.
\end{rem}

The functor from T1 topological abelian groups is not exact or even left exact, but we can focus on certain special cases.

\begin{prop}
	\label{prop-condses}
	Suppose  
	\[
		\begin{tikzcd}
			0 \arrow{r} & A \arrow{r}{f} & B \arrow{r}{g} & C \arrow{r} & 0
		\end{tikzcd}
	\]
	is a sequence of Hausdorff topological abelian groups that is exact as a sequence of abelian groups, $f$ is a closed map and for every compact $K \subseteq C$ there is a compact $B' \subseteq C$ such that $g(B') \supseteq K$. Then, the induced condensed sequence
	\[
		\begin{tikzcd}
			0 \arrow{r} & \underline{A} \arrow{r} & \underline{B} \arrow{r} & \underline{C} \arrow{r} & 0
		\end{tikzcd}
	\]
	is exact.
\end{prop}

\begin{proof}
	Applying the functor and analysing at every section, it is enough to show that 
	\begin{align*}
		0 \xrightarrow{} C(S,A) \xrightarrow{f \circ} C(S,B) \xrightarrow{g \circ} C(S,C) \xrightarrow{} 0
	\end{align*}
	is exact for every extremally disconnected set $S$. Injectivity of the first map is clear.

	Suppose $h \in C(S,B)$ and $gh = 0$. Then, $h(S) \subseteq \ker g = f(A)$. Define $l: S \to A$ by $l = f^{-1} \circ h$. This is continuous because $f$ is a closed map and $f \circ l = h$. This shows $\im (f \circ ) \supseteq \ker (g \circ )$.

	Since $g \circ f = 0$, the reverse inclusion is clear.

	Finally, we must show that $C(S,B)$ surjects onto $C(S,C)$. Consider a map $h: S \to C$. Then, $h(S)$ is compact so for some compact $B' \subseteq  B$, $g(B') \supseteq h(S)$. The image $h(S)$ is a compact subset of a Hausdorff space so it is closed and $g^{-1} (h(S))$ is closed.

	$B$ is Hausdorff as well, so $B'$ is closed and $g^{-1} (h(S)) \cap B'$ (a closed subset of a compact space) must be compact as well. $g^{-1} (h(S)) \cap B'$ surjects onto $h(S)$ and $S$ is projective in the category of compact Hausdorff spaces, so $h$ lifts to a map $S \to g^{-1}(h(S)) \cap B' \subseteq B$.
\end{proof}

\begin{prop}
	\label{prop-condcolim}
	The functor $M \mapsto \underline{M}$ commutes with direct sums of discrete groups, where we endow the direct sum of groups with the discrete topology.
\end{prop}

\begin{proof}
	Let $\{M_i\}$ be a collection of discrete spaces and let $M = \bigoplus M_i$. Looking at the problem section-wise, it suffices to prove that
	\begin{align*}
		\bigoplus C(S,M_i) \cong C(S,M)	
	\end{align*}
	where $M$ has the discrete topology. Let $f \in C(S,M)$. The space $S$ is compact so its image onto a discrete space must be finite. It is clear then that $f \in \bigoplus C(S,M_i)$ by taking the projection into each coordinate.

	Let $g \in \bigoplus C(S,M_i)$. We want to prove that the induced map $S \to M$ is continuous. Let $g_i$ be the $i$th coordinate of $g$. Then, $g_i \neq 0$ for finitely many $i$. Each $g_i : S \to M_i$ partitions $S$ into clopen sets (by taking preimages of singletons) and there must be finitely many of them by compactness. By taking intersections, we can refine this to a finite partition of clopen sets such that $g_i$ is constant on each of the clopen sets for all $i$. This induces a continuous map $S \to \bigoplus M_i$.
\end{proof}

\begin{prop}
	Let $(K_i)_{i \in I}$ be a collection of compact Hausdorff spaces. Then,
	\begin{align*}
		\underline{\prod_I K_i} = \prod_I \underline{K_i}
	\end{align*}
\end{prop}

\begin{rem}
	This applies whether the $K_i$ are topological spaces or groups.
\end{rem}

\begin{proof}
	For extremally disconnected $S$, it is enough to show that
	\begin{align*}
		C(S, \prod_I K_i) = \prod_I C(S, K_i)
	\end{align*}
	Let $p_j :\prod_I K_i \to K_j$ be the projection map. Let $f: S \to \prod_I K_i$. Then, $(p_i \circ f)_{i \in I} \in \prod_I C(S, K_i)$.

	Now let $(g_i)_{i \in I} \in \prod_I C(S,K_i)$. We claim that the natural function $g: S \to \prod_I K_i$ is continuous.

	Let $U \subseteq \prod_I K_i$ be closed. Then, 
	\begin{align*}
		g^{-1} (U) = \bigcap _{i \in I} (g_i^{-1} \circ p_i)(U)
	\end{align*}
	$p_i$ is a surjection from a compact space to a Hausdorff space, so it is a closed map. Hence, $p_i (U)$ is closed. By continuity of $g_i$, the preimage $(g_i^{-1} \circ p_i)(U)$ is closed and hence the intersection is closed. This concludes the proof.
\end{proof}

\subsection{Resolutions of condensed abelian groups}

Derived functors and the derived category will be built through projectives, since the category of condensed abelian groups has enough projectives (\cite{condensed} Theorem 2.2) but not any non-zero injectives (there is no published proof in the literature of this fact, but Peter Scholze provided a proof in a Math Overflow thread).

Projectives $\ZZ[\underline{S}]$ for an extremally disconnected set $S$ behave in a similar way to the projectives in the category of abelian groups and we will use this to build certain resolutions.

\begin{theorem}[\cite{condensed}, Theorem 4.5]
	Let $A$ be an abelian group. Then, there is a projective resolution of $A$
	\begin{align*}
		\cdots \to \bigoplus _{j=1}^{n_i} \ZZ[A^{r_{i,j}}] \to \cdots \to \bigoplus _{j=1}^{n_1} \ZZ[A^{r_{1,j}}] \to A \to 0
	\end{align*}
	that is functorial in $A$ and the $r_{i,j}$ are independent of the choice of $A$.
\end{theorem}

\begin{rem}
	We do not provide a description of the differentials but this will not be necessary for any of our computations. The first few can be found in \cite{cristalline} and are of the form
	\begin{align*}
		\ZZ[A^4] \oplus \ZZ[A^3] \oplus \ZZ[A^3] \oplus \ZZ[A^2] \oplus \ZZ[A] \to \ZZ[A^3] \oplus \ZZ[A^2] \to \ZZ[A^2] \to \ZZ[A] \to A
	\end{align*}
	for an arbitrary abelian group $A$ and \cite[Theorem 4.10]{condensed} shows that this may be extended to a resolution of the desired form.
\end{rem}

\begin{theorem}
	\label{thm-rescondab}
	Let $M$ be a condensed abelian group and let $r_{i,j}$ be defined as above. Then, there is a resolution of $M$ of the form
	\begin{align*}
		\cdots \to \bigoplus _{j=1}^{n_i} \ZZ[M^{r_{i,j}}] \to \cdots \to \bigoplus _{j=1}^{n_1} \ZZ[M^{r_{1,j}}] \to M \to 0
	\end{align*}
\end{theorem}

\begin{rem}
	This resolution is not necessarily projective, since the objects $\ZZ[M^{r_{i,j}}]$ might not be projective. $\ZZ[\underline{S}]$ is projective for an extremally disconnected set $S$ but $\ZZ[T]$ might not be for a general condensed set $T$.
\end{rem}

\begin{proof}
	Define $Z[M]$ as the presheaf $S \mapsto \ZZ[M(S)]$, so that the sheafification of $Z[M]$ is $\ZZ[M]$. Then, for any extremally disconnected set we have a resolution
	\begin{align*}
		\cdots \to \bigoplus _{j=1}^{n_i} Z[(M(S))^{r_{i,j}}] \to \cdots \to \bigoplus _{j=1}^{n_1} Z[(M(S))^{r_{1,j}}] \to M(S) \to 0
	\end{align*}
	$M \mapsto M(S)$ commutes with limits and colimits so we get a resolution
	\begin{align*}
		\cdots \to \bigoplus _{j=1}^{n_i} Z[M^{r_{i,j}}] (S) \to \cdots \to \bigoplus _{j=1}^{n_1} Z[M^{r_{1,j}}] (S) \to M(S) \to 0
	\end{align*}
	The functoriality of the resolution ensures that for every map $T \to S$ (and hence a map $M(S) \to M(T)$) of extremally disconnected sets, the resolutions of $M(S)$ and $M(T)$ form a commutative ladder diagram, so each of the maps induces a natural transformation. We get a sequence of presheaves
	\begin{align*}
		\cdots \to \bigoplus _{j=1}^{n_i} Z[M^{r_{i,j}}] \to \cdots \to \bigoplus _{j=1}^{n_1} Z[M^{r_{1,j}}] \to M \to 0
	\end{align*}
	This sequence is exact at every section so the sequence must be exact. We may now sheafify and we are done because sheafification is an exact functor.
\end{proof}

It is possible to explicitly construct the projective resolution of a condensed abelian group of the form $\ZZ[\underline{S}]$, where $S$ is a compact Hausdorff space.

\begin{theorem}
	\label{thm-projres}
	Let $S$ be a compact Hausdorff space and let $S_\bullet \to S$ be a simplicial hypercover of $S$ by extremally disconnected spaces. Then,
	\begin{align*}
		\dots \to \ZZ[\underline{S_1}] \to \ZZ[\underline{S_0}] \to \ZZ[\underline{S}] \to 0
	\end{align*}
	is a projective resolution of $\ZZ[\underline{S}]$.
\end{theorem}

\begin{rem}
	By condensing and applying the free condensed abelian group functor, we would get a simplicial condensed abelian group rather than a complex. However, we obtain the desired complex by extracting the Moore complex in the same way as the Dold-Kan correspondence functor.
\end{rem}

\begin{proof}
	This is a standard fact about the category of abelian sheaves on a site \cite[\href{https://stacks.math.columbia.edu/tag/01GF}{Tag 01GF}]{stacks-project}.

	Hence, the complex is exact. Finally, the $\ZZ[\underline{S_i}]$ are projective by \cref{prop-projcondab}.
\end{proof}

\subsection{Condensed cohomology}

We will be working with different kinds of cohomology throughout this essay, which will in some cases coincide.

\begin{defi}
	Let $S$ be a compact Hausdorff space. Then, we have the following ways of computing the cohomology $H^i (S, M)$ for some discrete group $M$.
	\begin{enumerate}
		\item Sheaf cohomology: the right derived functor of $\Gamma: \ab(S) \to \ab$ defined by $\Gamma (F) = F(S)$. Then, $H^i_{\text{sheaf}}(S,M) = R^i \Gamma (M)$, where $M$ is considered as the constant sheaf with value $M$.
		\item \v{C}ech cohomology: referred to as $\Hicech (S,M)$.
	\end{enumerate}
	They are isomorphic for compact Hausdorff $S$ \cite[Theorem 5.10]{faisceaux}.
\end{defi}

Condensed abelian groups allow us to give a third definition of cohomology.

\begin{defi}
	For a T1 space $S$ and an abelian group $M$, we define condensed cohomology as
	\begin{align*}
		H^i (S, M) = \ext^i_{\condab} (\ZZ[\underline{S}], \underline{M})
	\end{align*}
	where $M$ is given the discrete topology.

	We can also write more generally for a condensed abelian group $A$ and a condensed set $T$,
	\begin{align*}
		H^i (T,A) = \ext^i (\ZZ[T], A)
	\end{align*}
\end{defi}

\begin{rem}
	By \cref{thm-projres}, we have a projective resolution of the condensed abelian group $\ZZ[\underline{S}]$, for a compact Hausdorff $S$, using a simplicial hypercover $S_\bullet \to S$. The condensed cohomology $H^i (S, M)$ would then be computed by the complex $\Hom (\ZZ[\underline{S_i}], M) \cong M(S_i)$ by \cref{prop-projcondab}. This means that the complex
	\begin{align*}
		0 \to M(S_0) \to M(S_1) \to M(S_2) \to \dots
	\end{align*}
	has the groups $H^i(S,M)$ as its cohomology.
\end{rem}

\begin{theorem}[\cite{condensed}, Theorem 3.2, 3.3]
	\label{thm-condcohomology}
	Let $S$ be a compact Hausdorff space. Then,
	\begin{enumerate}
		\item $H^i (S, M) \cong \Hicech(S, M)$
		\item $H^i (S, \RR) \cong 0$ for $i > 0$ and $H^0 (S, \RR) \cong C(S, \RR)$.
	\end{enumerate}
\end{theorem}

\section{Enriched Hom functors}
\label{section-enrichedhom}

\subsection{Internal Hom for condensed sets}

The category of sheaves on a site $\mathcal{C}$ naturally has an enriched hom functor by considering
\begin{align*}
	{\mathscr{H}\kern -.5pt om} (\mathcal{F}, \mathcal{G}) (S) = \Hom(\mathcal{F}| _{\mathcal{C}_{/S}}, \mathcal{G}|_{\mathcal{C}_{/S}})
\end{align*}

where $\mathcal{C}_{/S}$ is the slice category. We provide the following more practical definition. A discussion of the equivalence between definitions in the presheaf category can be found in \cite{nlab-internalhom}.

In this subsection we study this construction for condensed sets as well as some applications.

\begin{prop}
	\label{prop-condsethom}
	The category of condensed sets has a closed symmetric monoidal structure, with the usual product as tensor product and internal hom given by
	\begin{align*}
		\condsethom (A,B)(S) = \homcondset (A \times \underline{S}, B)
	\end{align*}
\end{prop}

\begin{proof}
	This could be proved in terms of an abstract adjoint functor theorem, but instead we provide an explicit construction of the isomorphism that will be necessary in a later result.

	Let $M,N,P$ be condensed sets and let $g: P \to \condsethom(M,N)$. For extremally disconnected $S$, $(p,m) \in P(S) \times M(S)$ and $\operatorname{id}: S \to S$ the identity function, we can define
	\begin{align*}
		f_S (p,m) = (g_S (p))_S (m \times \operatorname{id})
	\end{align*}
	where the subscript of an extremally disconnected set on a sheaf morphism indicates the induced map on sections. This is a map $P \times M \to N$.

	In a more schematic way, this can be understood as
	\begin{align*}
		g \mapsto (p\times m  \mapsto (g(p)(m, \operatorname{id})))
	\end{align*}
	Let $f: P \times M \to N$. For extremally disconnected $S,T$, $p \in P(S)$ and $(m,s) \in M(T) \times C(T,S)$.
	\begin{align*}
		(g_S(p))_T (m,s) = f(s^*(p), m)
	\end{align*}
	where $s^*$ is the induced map $P(S) \to P(T)$. This is a map $P \to \condsethom(M,N)$.

	More schematically,
	\begin{align*}
		f \mapsto (p \mapsto (m \times s \mapsto f(s^* (p), m)))
	\end{align*}
	We claim that the above induces an isomorphism of hom sets.

	On one side we have for $f : P \times M \to N$, by composing the two above constructions,
	\begin{align*}
		f \mapsto (p \times m \mapsto f(\operatorname{id}^* (p), m)) 
	\end{align*}
	But clearly $\operatorname{id}^* (p) = p$ so this gives
	\begin{align*}
		f \mapsto (p \times m \mapsto f(p,m)) = f
	\end{align*}
	In the other direction, we have
	\begin{align*}
		g \mapsto (p \mapsto (m \times s \mapsto g(s^* (p))(m, \operatorname{id})))
	\end{align*}
	We only need to check that $g(s^* (p)) (m, \operatorname{id}) = g(p) (m , s)$.

	Writing this more precisely by making the extremally disconnected sets explicit, we need to prove $(g_T (s^* (p)))_T (m, \operatorname{id}) = (g_S (p))_T (m, s)$ for a continuous map $s: T \to S$. We will abuse notation and write $s^*$ for the two  maps $P(S) \to P(T)$ and $\Hom(M \times \underline{S}, N) \to \Hom (M \times \underline{T}, N)$. By the definition of a natural transformation we have a commutative diagram
	\[
		\begin{tikzcd}
			P(S) \arrow{r}{g} \arrow{d}{s^*} & \Hom (M \times \underline{S}, N) \arrow{d}{s^*} \\
			P(T) \arrow{r}{g} & \Hom (M \times \underline{T}, N)
		\end{tikzcd}
	\]
	This tells us that $(g_T (s^* (p)))_T (m, \operatorname{id}) = (g_S (p))_T s^*(m, \operatorname{id})$. Looking at the problem at the given section, we have $s^*$ inducing a map $M(T) \times C(T,T) \to M(T) \times C(T,S)$. This is clearly given by composition with $s$ so it follows that
	\begin{align*}
		(g_T (s^* (p)))_T (m, \operatorname{id}) = (g_S (p))_T s^* (m, \operatorname{id}) = (g_S (p))_T (m, s)
	\end{align*}
	as required.
\end{proof}

The following gives us a more direct way of computing this enriched hom functor.

\begin{prop}
	Let $A$ and $B$ be Hausdorff topological spaces, where $A$ is compactly generated. Then,
	\begin{align*}
		\condsethom (\underline{A}, \underline{B}) = \underline{C(A,B)}
	\end{align*}
	where $C(A,B)$ is given the compact-open topology.
\end{prop}

\begin{rem}
	We include a reminder that for spaces $X$ and $Y$, the compact-open topology on $C(X,Y)$ is defined as having subbasis the following sets. For $K \subseteq X$ compact and $U \subseteq Y$ open,
	\begin{align*}
		V(K,U) = \{ f \in C(X,Y) | f(K) \subseteq U\}
	\end{align*}
	If $Y$ is locally compact and Hausdorff we get a hom-tensor adjunction
	\begin{align*}
		C(X \times Y, Z) \cong C(X, C(Y,Z))
	\end{align*}
	where $C(Y,Z)$ is given the compact-open topology \cite{compact-open}.
\end{rem}

\begin{proof}
	Taking $S$-valued points for extremally disconnected $S$, we need an isomorphism
	\begin{align*}
		\homcondset(\underline{A} \times \underline{S}, \underline{B}) \to C(S,C(A,B))
	\end{align*}
	By the hom-tensor adjunction of the compact-open topology, $C(S, C(A,B)) \cong C(A \times S, B)$. This is then clear because condensation is fully faithful when restricted to compactly generated spaces.
\end{proof}

\subsection{Internal Hom for condensed abelian groups}

We can make an analogous construction with condensed abelian groups. The following is a general property of abelian sheaves on a site.

\begin{prop}
	\label{prop-condhomab}
	Condensed abelian groups $\condab$ form a closed symmetric monoidal category, where the tensor product $A \otimes B$ is given by the sheafification of
	\begin{align*}
		S \mapsto A(S) \otimes B(S)
	\end{align*}
	We also have an enriched hom functor $\condhom$ satisfying the adjunction
	\begin{align*}
		\Hom (P, \condhom(M, N)) = \Hom (P\otimes M, N)
	\end{align*}
	If we apply this to the free object $P = \ZZ[\underline{S}]$ for an extremally disconnected set $S$,
	\begin{align*}
		\condhom (M, N) (S) = \Hom(\ZZ[\underline{S}] \otimes M, N)		
	\end{align*}
	which allows us to characterise $\condhom$.
\end{prop}

With these enriched functors, we can derive an enriched version of the adjunction between the forgetful functor $\condab \to \condset$ and free condensed abelian groups.

\begin{prop}
	\label{prop-enrichedadjunction}
	Let $A$ be a condensed set and $B$ a condensed abelian group. Then, there is an isomorphism of condensed sets
	\begin{align*}
		\condhom(\ZZ[A], B) \cong \condsethom (A, B)
	\end{align*}
\end{prop}

\begin{proof}
	Taking $S$-valued points, we need an isomorphism
	\begin{align*}
		\homcondab(\ZZ[A \times \underline{S}], B) \to \homcondset (A \times \underline{S}, B)
	\end{align*}
	But this is clear by the adjunction between free condensed abelian groups and the forgetful functor $\condab \to \condset$.
\end{proof}

The closed symmetric monoidal structure on condensed sets allows us to prove the following result.

\begin{theorem}
	Let $A$ and $B$ be Hausdorff topological groups, with $A$ compactly generated. Then,
	\begin{align*}
		\condhom(\underline{A}, \underline{B}) \cong \underline{\Hom_{\ab}(A,B)}
	\end{align*}
	where $\homab(A,B)$ carries the compact-open topology.
\end{theorem}

\begin{proof}
	Taking $S$-valued points, we want to construct isomorphisms
	\begin{align*}
		\homcondab(\underline{A} \otimes \ZZ[\underline{S}], \underline{B}) \to C(S, \homab(A,B))
	\end{align*}
	Using the hom-tensor adjunction in condensed abelian groups,
	\begin{align*}
		\homcondab(\underline{A} \otimes \ZZ[\underline{S}], \underline{B}) \cong \homcondab(\underline{A}, \condhom(\ZZ[\underline{S}], \underline{B})) \xhookrightarrow{} \homcondset (\underline{A}, \condhom(\ZZ[\underline{S}], \underline{B}))
	\end{align*}
	Now we can apply \cref{prop-enrichedadjunction} and have an inclusion
	\begin{align*}
		\homcondab(\underline{A} \otimes \ZZ[\underline{S}], \underline{B}) \xhookrightarrow{} \homcondset(\underline{A}, \condhom(\ZZ[\underline{S}], \underline{B})) \cong \homcondset(\underline{A}, \condsethom(\underline{S}, \underline{B}))
	\end{align*}
	By the hom-tensor adjunction in condensed sets,
	\begin{align*}
		\homcondset(\underline{A}, \condsethom(\underline{S}, \underline{B})) \cong \homcondset(\underline{A \times S}, \underline{B}) \cong C(A \times S, B)
	\end{align*}
	where the last isomorphism is due to the fully faithful embedding of compactly generated spaces in condensed sets. Then, we have an inclusion
	\begin{align*}
		\homcondab(\underline{A} \otimes \ZZ[\underline{S}], \underline{B}) \xhookrightarrow{} C(A \times S, B)
	\end{align*}
	We claim that the induced maps $A \times S \to B$ are additive in $A$. Let $f: \underline{A} \otimes \ZZ[\underline{S}] \to \underline{B}$. By our construction, this induces a map of condensed abelian groups $\underline{A} \to \condhom(\ZZ[\underline{S}], \underline{B})$, which by the inclusion used is a map of condensed sets $\hat{f}: \underline{A} \to \condsethom(\underline{S}, \underline{B})$ that is additive in $\underline{A}$ at each section.

	By \cref{prop-condsethom}, this induces $\hat{g}: \underline{A} \times \underline{S} \to \underline{B}$. We now check that this is linear in $A$ at each section. The construction in \cref{prop-condsethom} gives that for profinite $T$, and $(a,s) \in \underline{A}(T) \times \underline{S}(T)$,
	\begin{align*}
		\hat{g}_T (a,s) = (\hat{f} (a))_T (s \times id)
	\end{align*}
	Since $\hat{f}$ was linear in $a$, it is easy to see that $\hat{g}$ is too.

	Hence, the induced maps are linear in $A$. By the hom-tensor adjunction of the compact-open topology, we have continuous maps $S \to C(A,B)$ such that the induced maps on $A$ are additive. Hence, the image must be contained in $\homab (A,B)$. This shows an inclusion
	\begin{align*}
		\homcondab(\underline{A} \otimes \ZZ[\underline{S}], \underline{B}) \xhookrightarrow{} C(S, \homab(A,B))
	\end{align*}
	Now we must show surjectivity. Let $g: S \to \homab(A,B)$. By the reverse construction, this induces a map $\hat{g}: \underline{A} \times \underline{S} \to \underline{B}$ that is additive in $\underline{A}$ at every section.

	Applying \cref{prop-condsethom}, this induces a map of condensed sets $\hat{f}: \underline{A} \to \condsethom(\underline{S}, \underline{B})$, such that for profinite $T_1$ and $T_2$, $a \in \underline{A}(T_1)$ and $(s,t) \in \underline{S}(T_2) \times \underline{T_1}(T_2)$,
	\begin{align*}
		(\hat{f}_{T_1} (a))_{T_2} (s,t) = \hat{g}_{T_2} (t^*(a), s)
	\end{align*}
	where $t^* : \underline{A}(T_1) \to \underline{A}(T_2)$ is the induced map by the sheaf. Since $\underline{A}$ is a condensed abelian group, $t^*$ must be additive in $A$, so $\hat{f}$ must be additive by additivity of $\hat{g}$.

	Hence, the map is linear on $\underline{A}$ at every section so this results in a map of condensed abelian groups $\underline{A} \to \condhom(\ZZ[\underline{S}], \underline{B})$. By adjunctions, this lifts to a map $\underline{A} \otimes \ZZ[\underline{S}] \to \underline{B}$ as required.
\end{proof}

\begin{rem}
	The above is proved in \cite[Proposition 4.2]{condensed} but we include an alternative proof. In the original argument, adjuctions are used to prove that a map $\underline{A} \otimes \ZZ[\underline{S}] \to \underline{B}$ induces a map $A \times S \to B$. More precisely, the argument used in these notes is
	\begin{align*}
		\homcondab (\underline{A} \otimes \ZZ[\underline{S}], \underline{B}) \xhookrightarrow{} \homcondab (\ZZ[\underline{A} \times \underline{S}], \underline{B}) \cong C(A \times S, B)
	\end{align*}
	However, it did not seem a priori clear to the author why this map should be additive in $A$ to conclude a map $S \to \homab (A,B)$. Due to the adjunctions used, both in the tensor product and the free condensed abelian group, it is not obvious that this additivity should be preserved. Sheafification makes the explicit maps somewhat uncertain. This led to the construction of the internal Hom in condensed sets for this essay.
\end{rem}

\subsection{The derived category of condensed abelian groups}

With the previous constructions, we can also build a derived category. Within the conventional theory of $1$-categories, it is not known whether $D(\condab) \cong \operatorname{Cond}(D(\ab))$. There is no reason why $\operatorname{Cond}(D(\ab))$ should be triangulated.

However, this can all be solved using $\infty$-categories. Sheaves of abelian groups on the category of extremally disconnected sets are very well-behaved since homology commutes with evaluation on extremally disconnected sets. This results in a hypercomplete $(\infty,1)$-topos which gives an equivalence between the desired categories \cite[Theorem 2.1.2.2]{saglurie}.

Nevertheless, this advanced treatment will not be necessary for the purposes of this essay and we will treat $D(\condab)$ as the usual triangulated category, which allows essentially the same computations. Whenever we talk about distinguished triangles, the $\infty$-categorical equivalent object is a cofiber sequence.

\begin{theorem}[\cite{condensed}]
	The derived category of condensed abelian groups $D(\condab)$ is closed symmetric monoidal, where we have an adjunction between an internal hom and a total tensor product
	\begin{align*}
		\Hom(M \otimes ^L N, P) \cong \Hom(M, \condrhom(N,P))
	\end{align*}
\end{theorem}

\begin{example}
	\label{example-condrhom}
	We will explain in more detail how to calculate $\condrhom$, as the derived functor of $\condhom^\bullet$ in the derived category. Let $A$ and $B$ be two bounded above cochain complexes of condensed abelian groups. 

	We now take the double complex $\{\condhom (A^p, B^{-q})\}$ with maps $A^p \to A^{p+1}$ and $B^{-q} \to B^{-q -1}$, which by the functoriality of $\condhom$ turn into differentials
	\begin{align*}
		d^h : \condhom(A^p, B^{-q}) \to \condhom(A^{p+1},B^{-q}) \\
		d^v : \condhom(A^p, B^{-q}) \to \condhom(A^p, B^{-q-1})
	\end{align*}
	(one needs to adjust the sign of $d^v$ multiplying by $(-1)^{p+q+1}$ to ensure the antisymmetry of the differentials). Define
	\begin{align*}
		\condhom ^\bullet (A,B)= \tot^\Pi (\{\condhom(A^p, B^{-q})\})
	\end{align*}
	$\condhom^\bullet(-,B)$ is a morphism of triangulated categories (when considering the homotopy category) so we may construct the total derived functor $\condrhom(-,B)$. Due to the lack of injectives in condensed abelian groups it is not possible to do an analogous construction with $\condhom^\bullet (A,-)$. 

	We then obtain a bifunctor $\condrhom : D(\condab)^{op} \times D(\condab) \to D(\condab)$. This is completely analogous to the usual $\rhom$ and details may be found in \cite{weibel}.
	
	$\condab$ has enough projectives, so we can find a Cartan-Eilenberg resolution of projectives of $A^{op}$ ($A$ considered as a chain complex so it is bounded below) and take the chain complex $\tot ^\oplus$, which is quasi-isomorphic to $A^{op}$ (\cite{weibel} \S 5.7). Since $A^{op}$ is bounded below, its Cartan-Eilenberg resolution is a first quadrant bicomplex, so when applying $\tot ^\oplus$ we get a bounded below complex. We now take again the opposite cochain complex to get a complex of projectives quasi-isomorphic to $A$, call it $P$. Then,
	\begin{align*}
		\condrhom(A,B) = \condhom^\bullet (P, B)
	\end{align*}
	However, this work will focus on computing $\condrhom(A,B)$ where $A$ and $B$ are condensed abelian groups and we are just considering them as cochain complexes concentrated on degree zero. Then, we may substitute $A$ by just a projective resolution.

	Observe then that the Cartan-Eilenberg resolution would be concentrated on the vertical axis, and $\condrhom(A,B)$ would be the functor $\condhom (-,B)$ applied to the projective resolution of $A$.
\end{example}

\begin{rem}
	The adjunction gives the usual properties of $\rhom$
	\begin{enumerate}
		\item $\condrhom(\bigoplus A_i, B) = \prod \condrhom (A_i, B)$
		\item $\condrhom(A, \prod B_j) = \prod \condrhom(A, B_j)$
	\end{enumerate}
\end{rem}

\section{Derived mapping spaces}
\label{section-derivedmapping}

In this section we explore the internal hom $\condrhom$ of the derived category.

\subsection{The spectral sequence argument}

Finding a projective resolution of condensed abelian groups is not a trivial calculation. That is why we provide this construction that uses spectral sequences and condensed cohomology.

\begin{theorem}
	\label{thm-spectral}
	Let $A$ and $M$  be condensed abelian groups and $S$ extremally disconnected. Then, there is a spectral sequence
	\begin{align*}
		E_1^{pq} = \prod_{j=1}^{n_p} H^q (A^{r_{p,j}} \times \underline{S}, M) \implies \condext^{p+q} (A[0],M[0])(S)
	\end{align*}
\end{theorem}

\begin{proof}
	By \cref{prop-condhomab},
	\begin{align*}
		\condhom (A, M) (S) = \Hom (A \otimes \ZZ[\underline{S}], M)
	\end{align*}
	Then we must have
	\begin{align*}
		\condrhom (A, M) (S) = \rhom (A \otimes \ZZ[\underline{S}], M)
	\end{align*}
	so we will be interested in studying the latter object to understand the sections of $\condrhom (A, M) (S)$. This same reasoning will apply to ext groups because evaluation at extremally disconnected sets commutes with homology, so
	\begin{align*}
		\condext(A,M)(S) = \condext (A \otimes \ZZ[\underline{S}], M)
	\end{align*}
	By \cref{thm-rescondab}, we may pick a resolution of the form
	\begin{align*}
		\cdots \to \bigoplus _{j=1}^{n_i} \ZZ[A^{r_{i,j}}] \to \cdots \to \bigoplus _{j=1}^{n_1} \ZZ[A^{r_{1,j}}] \to A \to 0
	\end{align*}
	This is not necessarily projective, but we obtain a complex of free condensed abelian groups quasi-isomorphic to $A[0]$.

	Since $\ZZ[\underline{S}]$ is projective, it is flat so we may tensor by it to obtain a resolution of the form
	\begin{align*}
		\cdots \to \bigoplus _{j=1}^{n_i} \ZZ[A^{r_{i,j}} \times \underline{S}] \to \cdots \to \bigoplus _{j=1}^{n_1} \ZZ[A^{r_{1,j}} \times \underline{S}] \to A  \otimes \ZZ[\underline{S}] \to 0
	\end{align*}
	Consider a Cartan-Eilenberg resolution $P^{**}$ of the above complex so $\tot^\oplus (P^{**})$ is a projective resolution of $A \otimes \ZZ[\underline{S}]$. By \cref{example-condrhom},
	\begin{align*}
		\rhom(A \otimes \ZZ[\underline{S}], M) = \Hom (\tot^\oplus (P^{**}), M)
	\end{align*}
	where we apply $\Hom(-,M)$ on each degree. However, $\Hom(-,M)$ commutes with finite products so
	\begin{align*}
		\rhom(A \otimes \ZZ[\underline{S}], M) = \tot^\oplus \Hom(P^{**}, M)
	\end{align*}
	To compute $\ext^i (A \otimes \ZZ[\underline{S}], M)$ we may just take the cohomology of the total complex. For this computation, we will use spectral sequences.

	Take the spectral sequence corresponding to the double complex, induced by the trivial filtration on columns. We then have that the first iteration is equal to the vertical cohomology. However, columns in a Cartan-Eilenberg resolution are projective resolutions so
	\begin{align*}
		E_1^{pq} \cong \ext^q (\bigoplus_{j=1}^{n_p} \ZZ[A^{r_{p,j}} \times \underline{S}], M) \cong \prod_{j=1}^{n_p} \ext^q (\ZZ[A^{r_{p,j}} \times \underline{S}], M) \cong \prod_{j=1}^{n_p} H^q (A^{r_{p,j}} \times \underline{S}, M)
	\end{align*}
	The spectral sequence converges to the cohomology of the total complex so
	\begin{align*}
		E_1^{pq} \implies \ext^q (A \otimes \ZZ[\underline{S}], M) \cong \condext^q (A,M)(S)
	\end{align*}
\end{proof}

\subsection{A few important cases}

The following result is analogous to the equivalent calculation in the usual category of abelian groups.

\begin{lemma}
	\label{lemma-rhomZZ}
	For any condensed abelian group $M$,
	\begin{align*}
		\condrhom (\underline{\ZZ}, M ) = M [0]
	\end{align*}
\end{lemma}

\begin{proof}
	First we will prove that
	\begin{align*}
		\underline{\ZZ} = \ZZ [\underline{*}]
	\end{align*}
	Notice that $\underline{*}$ is the zero sheaf, so $\ZZ[\underline{*}]$ is the locally constant sheaf with value $\ZZ$, which is exactly the sheaf of constant functions with codomain $\ZZ$, so we get the desired equality.

	Hence, $\underline{\ZZ}$ is projective. So to consider $\condrhom$ we only need to look at $\condhom(\underline{\ZZ}, M)$.  For an extremally disconnected set $S$,
	\begin{align*}
		\condhom(\ZZ[\underline{*}], M)(S) = \Hom (\ZZ[\underline{*}] \otimes \ZZ[\underline{S}], M) = \Hom(\ZZ[\underline{*} \times \underline{S}], M) = \Hom(\ZZ[\underline{S}], M) = M(S)
	\end{align*}
	So equality follows.
\end{proof}

\begin{theorem}
	Let $M$ be a discrete abelian group. Then, $\condrhom(\underline{\RR}, \underline{M}) = 0$.
\end{theorem}

\begin{proof}
	By \cref{thm-spectral} we obtain spectral sequences for $\condext^q (\condRR, \underline{M})$ and $\condext^q (0,\underline{M}) = 0$. Hence, to prove the result it would suffice to prove that the map
	\begin{align*}
		H^q (\RR^r \times S, M) \to H^q (S, M)
	\end{align*}
	induced by pullback of $0 \times \operatorname{id}: S \to \mathbb{R}^r \times S$	is an isomorphism so they have isomorphic spectral sequences and hence converge to the same cohomology.

	We now may write $\RR^r$ as the filtered colimit of compact spaces. 
	\begin{align*}
		\RR^r = \lim_{\longrightarrow} [-n,n]^r
	\end{align*}
	We claim that this filtered colimit commutes with the condensation functor.

	Looking at it section wise (since condensed sets are stable under pointwise filtered colimits) it is enough to prove that $C(T, \RR^r) = \colim C(T, [-n,n]^r)$ for extremally disconnected $T$. The space $T$ is compact, so any function $T \to \RR^r$ is contained in a compact set of the form $[-n,n]^r$ and equality follows.

	The functor of free condensed abelian groups is a left adjoint, so it preserves colimits. Hence,
	\begin{align*}
		\ZZ[\underline{\RR^r}] = \lim_{\longrightarrow} \ZZ[\underline{[-n,n]^r}]
	\end{align*}
	We now write this as a homotopy colimit. $H^q(\RR^r \times S, M)$ is the cohomology of $\rhom (\ZZ[\underline{\RR^r}\times \underline{S}], \underline{S})$ and $\rhom$ takes homotopy colimits to homotopy limits so
	\begin{align*}
		\rhom(\ZZ[\underline{\RR^r} \times \underline{S}], M) = \holim \rhom(\ZZ[\underline{[-n,n]^r} \times \underline{S}], M) 	
	\end{align*}
	The cohomology of the complex on the right is $H^q([-n,n]^r\times , M)$ and by pullback of $0 \times \operatorname{id} : S \to [-n,n]^r \times S$ we get a map on the cohomologies
	\begin{align*}
		H^q([-n,n]^r \times S, M) \to H^q(S, M)
	\end{align*}
	This is an isomorphism for every $q$, since by \cref{thm-condcohomology} this corresponds to sheaf cohomology which is homotopy invariant. Hence, the map is a quasi-isomorphism between complexes. This is true for every $n$, so the homotopy limit must also be a weak equivalence which gives us an isomorphism of cohomologies. Then, the map
	\begin{align*}
		H^q (\RR^r \times S, M) \to H^q (S, M)
	\end{align*}
	is an isomorphism.
\end{proof}

\begin{theorem}[\cite{condensed}, Theorem 4.3]
	Let $M$ be a discrete abelian group and let $A = \prod_I \RR/\ZZ$. Then,
	\begin{align*}
		\condrhom(\underline{A}, \underline{M}) = \bigoplus_I \underline{M}[-1]
	\end{align*}
\end{theorem}

\begin{proof}
	We have a short exact sequence
	\[
		\begin{tikzcd}
			0 \arrow{r} & \underline{\ZZ} \arrow{r} & \underline{\RR} \arrow{r} & \underline{\RR}/\underline{\ZZ} \arrow{r} & 0	
		\end{tikzcd}
	\]
	because the corresponding short exact sequence of abelian groups satisfies the hypothesis of \cref{prop-condses}.

	Since the functor $\condrhom(-,\underline{M})$ is a morphism of triangulated categories, $\condrhom (\condZZ, \underline{M}) = \underline{M}$ and $\condrhom(\condRR, \underline{M}) = 0$, we just complete the distinguished triangle and obtain
	\begin{align*}
		\condrhom (\condRRZZ, \underline{M}) = \underline{M}[-1]
	\end{align*}
	It then would suffice to show that
	\begin{align*}
		\lim_{\longrightarrow} \condrhom (\prod_J \condRRZZ, \underline{M}) = \condrhom(\prod_I \condRRZZ, \underline{M})
	\end{align*}
	where $J$ runs over all finite subsets of $I$. Applying \cref{thm-spectral} to make an argument with spectral sequences, it suffices to prove that the natural map
	\begin{align*}
		\lim_{\longrightarrow} H^i ( S \times \prod_J \RR/\ZZ , M) \to H^i (S \times \prod_I \RR/\ZZ, M)
	\end{align*}
	is an isomorphism. But this follows from comparison with sheaf cohomology \cite[Lemma 4.1.3]{topos}.
\end{proof}

\begin{theorem}[\cite{condensed}, Theorem 4.3]
	Let $A$ be a compact abelian group. Then,
	\begin{align*}
		\condrhom (\underline{A}, \condRR) = 0
	\end{align*}
\end{theorem}

\begin{proof}
	Apply \cref{thm-spectral} to obtain a spectral sequence $E^{pq}$ with
	\begin{align*}
		E^{p,q}_1 = \prod_{j=1}^{n_p} H^q (A^{r_{p,j}} \times S, \RR) \implies \condext ^{p+q}(\underline{A}, \condRR)(S)
	\end{align*}
	$A^{r_{p,j}} \times S$ is compact and by \cref{thm-condcohomology}, $H^q (K, \RR)$ vanishes for compact $K$ and $q > 0$, and $H^0(K, \RR) = C(K, \RR)$. Hence, all of the spectral sequence vanishes except for the horizontal axis. Then, it suffices to show that the spectral sequence vanishes entirely in the next iteration. We will prove that the vertical axis
	\begin{align*}
			0 \to  \bigoplus_{j=1}^{n_0} C(A^{r_{0,j}} \times S, \RR) \to  \bigoplus_{j=1}^{n_1} C(A^{r_{1,j}} \times S, \RR) \to \dots
	\end{align*}
	is exact.

	Call $F(A)_\bullet$ the  resolution of $A$ used to obtain the spectral sequence. Let $2: F(A) \to F(A)$ be the chain map where every component is multiplied by $2$, and let $[2]: F(A) \to F(A)$ be the map induced by multiplying $A$ by $2$ (this is possible since the resolution is functorial in $A$).

	\cite{condensed} shows that these two maps are homotopic (Proposition 4.17). Let $h_{\bullet}: F(A) \to F(A)$ be such a homotopy. Then, $2 - [2] = dh + hd$.

	Suppose $f \in \bigoplus_{j=1}^{n_i} C(A^{r{i,j}} \times S, \RR)$ satisfies $df = 0$. Then,
	\begin{align*}
		2f - [2]^*(f) = d(h_{i-1}^* (f)) + h_i^*(d(f)) = d(h_{i-1}^* (f)) \implies f = \frac{1}{2}[2]^* (f) + d(\frac{1}{2} h_{i-1}^* (f))
	\end{align*}
	Inductively we obtain that
	\begin{align*}
		f = \frac{1}{2^n} [2^n]^* (f) + d \left( \sum_{k=1}^n \frac{1}{2^k} h_{i-1}^*([2^{n-1}]^* (f)) \right)
	\end{align*}
	$[2^n]^* (f)$ is the map induced on $f$ when $A$ is multiplied by $2$. This does not change the image so it is bounded with $\| [2^n]^* (f)\| \leq \| f \|$. Also, $h_{i-1}^*$ is a map of Banach spaces so it must be bounded. Hence, the sum converges and we can write
	\begin{align*}
		f = d \left( \sum_{k=1}^\infty \frac{1}{2^k} h_{i-1}^*([2^{n-1}]^*(f)) \right)
	\end{align*}
	so the complex is exact.
\end{proof}

\subsection{Locally compact abelian groups}

We wil focus on a special class of topological groups, since they have a set structure that simplifies calculations.

\begin{defi}
	A topological space is \emph{locally compact} if every point is contained in a neighbourhood that is contained in some compact set.

	A \emph{locally compact abelian group} is a Hausdorff abelian group whose underlying topology is locally compact.
\end{defi}

The following results may be found in \cite[Proposition 2.2]{lcagroups}.

\begin{theorem}[Structure of LCA groups]
	\label{thm-lcastructure}
	Let $A$ be a locally compact abelian group.
	\begin{enumerate}
		\item $A \cong \RR^n \times A'$, where $A'$ is an extension of a discrete abelian group by a compact abelian group.
		\item The Pontryagin duality functor $A \mapsto \DD(A) = \Hom (A, \RR/\ZZ)$. $\DD(A)$ takes values in locally compact abelian groups and $\DD(\DD(A)) \cong A$.
		\item If $A$ is compact (resp. discrete) if and only if $\DD(A)$ is discrete (resp. compact).
	\end{enumerate}
\end{theorem}

\subsection{RHom in locally compact abelian groups}

The rest of this section will consist of simplifying the necessary calculations for $\condrhom$ of locally compact abelian groups.

\begin{prop}
	\label{prop-rhomlcastruct}
	Let $M$ and $N$ be LCA groups. Then, $\condrhom (\underline{M}, \underline{N})$ reduces to assuming $M,N = \RR$ or an extension of a discrete abelian group by a compact abelian group.
\end{prop}

\begin{proof}
	By \cref{thm-lcastructure}, $M \cong \RR^m \times M'$ where $M$ is an extension of a discrete abelian group by a compact abelian group. Similar for $N$. We are done now observing that $\condrhom(-,-)$ commutes with finite products (= direct sums) in either argument.
\end{proof}

\begin{prop}
	\label{prop-rhomcompdisc}
	$\condrhom(\underline{M}, \underline{N})$ where $M$ or $N$ is an extension of a discrete abelian group by a compact abelian group reduces to $M$ or $N$ is compact or discrete.
\end{prop}

\begin{proof}
	Suppose $M$ is an extension of a discrete abelian group by a compact abelian group. Then, it fits in a short exact sequence
	\[
		\begin{tikzcd}
			0 \arrow{r} & A \arrow{r}{f} & M \arrow{r}{g} & B \arrow{r} & 0
		\end{tikzcd}
	\]
	where $A$ is compact and $B$ is discrete.

	We are going to prove that the sequence
	\[
		\begin{tikzcd}
			0 \arrow{r} & \underline{A} \arrow{r} & \underline{M} \arrow{r} & \underline{B} \arrow{r} & 0
		\end{tikzcd}
	\]
	is exact. By \cref{prop-condses} it is enough to show that $f$ is a closed map and that for every compact subset $B' \subseteq B$ there is a compact subset $M' \subseteq M$ such that $g(M') \supseteq  B'$.

	$f$ is a continuous map from a compact space to a Hausdorff space, so it must be closed. On the other hand, since $B$ is discrete, its only compact subsets are finite, which are covered by finite (and hence compact) subsets of $M$.

	This short exact sequence of condensed abelian groups induces a distinguished triangle
	\begin{align*}
		\underline{A} [0] \to \underline{M}[0] \to \underline{B}[0] \to \underline{A}[1]
	\end{align*}
	$\condrhom: D(\condab)^{op} \times D(\condab)$ is a morphism of triangulated categories when fixing either argument, so we see that making the calculation for discrete and compact is enough and then we may just complete the triangle. Similar if $N$ is an extension of a discrete abelian group by a compact abelian group.
\end{proof}

\begin{prop}
	\label{prop-rhomMdisc}
	$\condrhom (\underline{M}, \underline{N})$ where $M$ is discrete reduces to $M = \ZZ$.
\end{prop}

\begin{proof}
	A discrete group $M$ has a two-term free resolution
	\[
		\begin{tikzcd}
			0 \arrow{r} & \ZZ[Q] \arrow{r} & \ZZ[M] \arrow{r} & M \arrow{r} & 0
		\end{tikzcd}
	\]
	where $\ZZ[Q]$ is the kernel of the obvious map $\ZZ[M] \to M$. Since they are all discrete groups, the map $\ZZ[Q] \to \ZZ[M]$ is closed. Also, any compact subset of $M$ must be finite so it is covered by a finite (and hence compact) subset of $\ZZ[M]$. We may now apply \cref{prop-condses} and see that the condensed sequence is exact.

	We can now write $\ZZ[Q] = \bigoplus_Q \ZZ$ as a sum. Then, by \cref{prop-condcolim}, $\underline{\ZZ[Q]} = \bigoplus _Q \condZZ$ so
	\begin{align*}
		\condrhom (\underline{\ZZ[Q]}, \underline{N}) = \condrhom ( \bigoplus_Q \underline{\ZZ}, \underline{N}) = \prod_Q  \condrhom (\underline{\ZZ}, \underline{N})
	\end{align*}
	Similar for $\condrhom (\underline{\ZZ[M]}, \underline{N})$, so we may just complete the triangle to compute $\condrhom(\underline{M}, \underline{N})$.
\end{proof}

\begin{prop}
	\label{prop-rhomMcompact}
	$\condrhom(\underline{M}, \underline{N})$ where $M$ is compact reduces to $M = \prod_I \underline{\RR/\ZZ}$ for some set $I$.
\end{prop}

\begin{proof}
	We will use Pontryagin duality. Since $M$ is compact, by the structure result of locally compact abelian groups, $\DD(M)$ is discrete. This then gives a two-term free resolution
	\[
		\begin{tikzcd}
			0 \arrow{r} & \ZZ[Q] \arrow{r} & \ZZ[\DD(M)] \arrow{r} & \DD(M) \arrow{r} & 0
		\end{tikzcd}
	\]
	Pontryagin duality is an exact functor (it is an equivalence of categories) so, noting that $\DD(\ZZ) = \RR/\ZZ$ and sums factor as products, we have a short exact sequence
	\[
		\begin{tikzcd}
			0 \arrow{r} & M \arrow{r} & \prod_I \RR/\ZZ \arrow{r} & \prod_J \RR/\ZZ \arrow{r} & 0
		\end{tikzcd}
	\]
	for some sets $I, J$. Now we will prove that the condensed sequence is exact.
	
	The first map is from a compact set to a Hausdorff space, so it is closed. Finally, notice that a compact subset of $\prod_J \RR/\ZZ$ is closed so its preimage is a closed subset of $\prod_I \RR/\ZZ$, which must be compact. We may now apply \cref{prop-condses}.

	Hence, the induced sequence is exact and we may again get a distinguished triangle.
\end{proof}

\begin{prop}
	\label{rhom-Mcompact}
	$\condrhom(\underline{M},\underline{N})$ where $N$ is compact reduces to $N = \RR/\ZZ$.
\end{prop}

\begin{proof}
	Use the same reasoning with Pontryagin duality and that products commute with $\condrhom$ in the second argument.
\end{proof}

\begin{prop}
	\label{prop-rhomRses}
	We have the following similar cases.
	\begin{enumerate}
		\item $\condrhom(\underline{M}, \underline{N})$ where $M = \RR$ reduces to $M = \ZZ$ and $M = \RR/\ZZ$.
		\item $\condrhom(\underline{M}, \underline{N})$ where $N = \RR/\ZZ$ reduces to $N = \RR$ and $N$ discrete.
	\end{enumerate}
\end{prop}

\begin{proof}
	Use the short exact sequence
	\[
		\begin{tikzcd}
			0 \arrow{r} & \underline{\ZZ}  \arrow{r} & \underline{\RR} \arrow{r} & \underline{\RR/\ZZ} \arrow{r} & 0
		\end{tikzcd}
	\]
\end{proof}

With the above, we can finally compute $\condrhom$ for all locally compact abelian groups. 

\begin{theorem}
	$\condrhom (\underline{M}, \underline{N})$ where $M$ and $N$ are locally compact abelian groups reduces to the cases $M = \prod_I \RR/\ZZ$ for some set $I$ or $M = \ZZ$ and $N = \RR$ or $N$ discrete.
\end{theorem}

Observe that all of these cases were dealt with in the previous section.

The following diagrams summarise the simplifications made for $\condrhom (\underline{M}, \underline{N})$.
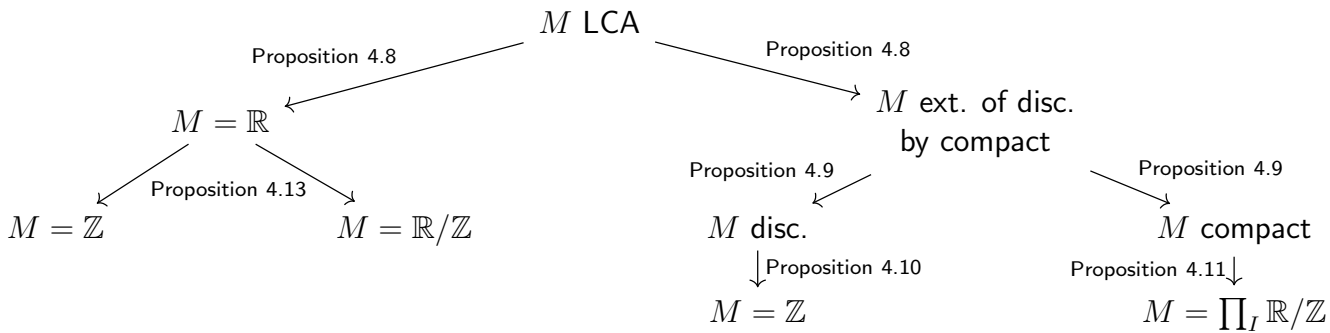
\[
	\begin{tikzcd}[sep=small]
		& & & \text{$M$ LCA} \arrow[lld, "\text{\cref{prop-rhomlcastruct}}", swap]  \arrow{rrd}{\text{\cref{prop-rhomlcastruct}}} & & & \\
		& M = \RR \arrow[ld, "\text{\cref{prop-rhomRses}}"] \arrow{rd} & & & & \begin{matrix} \text{$M$ ext. of disc.} \\ \text{by compact} \end{matrix} \arrow[ld, "\text{\cref{prop-rhomcompdisc}}", swap] \arrow[rd, "\text{\cref{prop-rhomcompdisc}}"] & \\
		M = \ZZ & & M = \RR/\ZZ & & \text{$M$ disc.} \arrow[d, "\text{\cref{prop-rhomMdisc}}"] & & \text{$M$ compact} \arrow[d, "\text{\cref{prop-rhomMcompact}}", swap]\\
		& & & & M = \ZZ &  & M = \prod_I \RR/\ZZ 
	\end{tikzcd}
\]

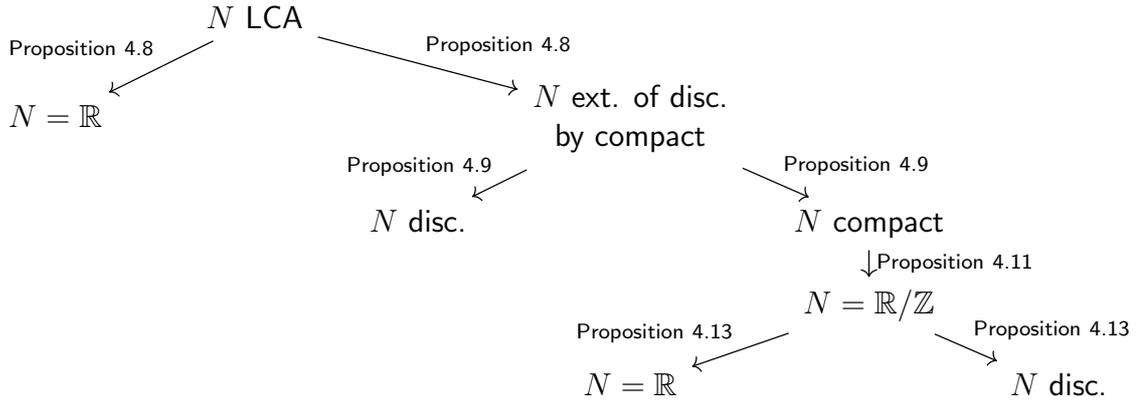
\[
	\begin{tikzcd}[sep=small]
		& & & \text{$N$ LCA} \arrow[lld, "\text{\cref{prop-rhomlcastruct}}", swap] \arrow[rrd, "\text{\cref{prop-rhomlcastruct}}"] & & & & \\
		& N = \RR & & & & \begin{matrix} \text{$N$ ext. of disc.} \\ \text{by compact} \end{matrix} \arrow[ld, "\text{\cref{prop-rhomcompdisc}}", swap] \arrow[rd, "\text{\cref{prop-rhomcompdisc}}"] & & \\
		& & & & \text{$N$ disc.} & & \text{$N$ compact} \arrow[d, "\text{\cref{prop-rhomMcompact}}"] & \\
		& & & & & & N = \RR/\ZZ \arrow[ld, "\text{\cref{prop-rhomRses}}", swap] \arrow[rd, "\text{\cref{prop-rhomRses}}"] \\
		& & & & & N = \RR & & \text{$N$ disc.}
	\end{tikzcd}
\]

This leads to the following interesting corollary. By following through all the computations in the previous lemmas one can show that higher $\underline{\ext}$ groups vanish (c.f. \cite{condensed}). This means that in the following examples we will focus on the first two ext groups.

\begin{cor}
	\label{cor-vanishext}
	If $M$ and $N$ are locally compact abelian groups,
	\begin{align*}
		\underline{\ext}^i (\underline{M}, \underline{N}) = 0
	\end{align*}
	for $i \geq 2$.
\end{cor}

\begin{rem}
	This seems to suggest that condensed abelian groups might have dimension $2$. However, we remark that this is not true.

	By \cref{thm-condcohomology}, for a compact Hausdorff space $S$,
	\begin{align*}
		\ext^i (\ZZ[\underline{S}], \condZZ) = H^i (S, \ZZ) \cong H^i _{\text{sheaf}} (S, \ZZ)
	\end{align*}
	which does not always vanish for $i \geq 2$ so it is impossible that we have projective dimension $2$.
\end{rem}

\subsection{Examples}

\begin{example}
	With the above one can obtain rather unsurprisingly that
	\begin{align*}
		\condrhom(\underline{\RR}, \underline{\RR^{\text{disc}}}) = 0
	\end{align*}
	since $\condrhom (\underline{\ZZ}, \underline{\RR^{\text{disc}}}) = \underline{\RR^{\text{disc}}}$ and $\condrhom(\underline{\RR/\ZZ}, \underline{\RR^{\text{disc}}}) = \underline{\RR^{\text{disc}}}[-1]$ so we may just complete the distinguished triangle induced from the exact sequence $0 \to \ZZ \to \RR \to \RR/\ZZ \to 0$.
\end{example}

\begin{example}
	\label{example-extRRZZ}
	One can also prove that
	\begin{align*}
		\condhom (\condQQdisc, \condRRZZ) = \underline{\hat{\QQ}} \\
		\condext^1 (\condQQdisc, \condRRZZ) = 0
	\end{align*}
	where $\hat{\QQ}$ is the solenoid, the Pontryagin dual of $\QQdisc$.

	The first part is clear. We may now use again a two-term resolution of condensed abelian groups. The long exact sequence on $\ext$ groups gives
	\begin{align*}
		0 \to \hat{\QQ} \to \widehat{\ZZ[\QQ]} \to \widehat{\ZZ[K]} \to 0
	\end{align*}
	All these groups are Pontryagin duals of discrete spaces so they are compact. Hence, the first map is from a compact space to a Hausdorff space, so it is closed. Any compact subset of $\widehat{\ZZ[K]}$ is closed so it has a closed and therefore compact preimage in $\widehat{\ZZ[\QQ]}$. By \cref{prop-condses} the condensed sequence is exact.

	Using the long exact sequence for enriched $\ext$ groups,
	\begin{align*}
		0 \to \underline{\hat{\QQ}} \to \underline{\widehat{\ZZ[\QQ]}} \to \underline{\widehat{\ZZ[K]}} \to \condext^1 (\condQQdisc, \condRRZZ) \to 0
	\end{align*}
	But due to the condensed short exact sequence, $\underline{\widehat{\ZZ[\QQ]}} \to \underline{\widehat{\ZZ[K]}}$ is a surjection and
	\begin{align*}
		\condext^1 (\condQQdisc, \condRRZZ) = 0
	\end{align*}
\end{example}

\begin{example}
	Suppose we want to calculate $\condrhom(\condRRdisc, \condRR)$.
	
	By Zorn's lemma, choose a basis $\mathcal{B}$ of $\RR$ over $\QQ$. \cite{Neumann} gives explicit examples of algebraically independent subsets of $\RR$ over $\QQ$ of cardinality $2^{\aleph_0}$ and $2^{\aleph_0}$ is also a trivial upper bound for $|\mathcal{B}|$ so $|\mathcal{B}| = 2^{\aleph_0}$. Then,
	\begin{align*}
		\condrhom(\condRRdisc, \condRR) = \condrhom (\bigoplus _\mathcal{B} \condQQdisc, \condRR) = \prod_\mathcal{B} \condrhom( \condQQdisc, \condRR) = \condrhom(\condQQdisc, \condRR)^{2^{\aleph_0}}
	\end{align*}
	We first have that
	\begin{align*}
		\condext ^0 (\condQQdisc, \condRR) = \condhom (\condQQdisc, \condRR) = \underline{\Hom (\QQdisc, \RR)}
	\end{align*}
	Any function $\QQdisc \to \RR$ is continuous so we only need to look at group homomorphisms, which are functions satisfying $f(x+y) = f(x) + f(y)$ for all $x,y \in \QQdisc$.

	This is exactly Cauchy's functional equation, whose only solutions are of the form $f(x) = \alpha x$ for some $\alpha \in \RR$ \cite[Chapter 5]{Cauchy-eq}, so the underlying set is $\RR$. Now we must compute the topology.

	We claim that the compact-open topology gives $\RR$ its natural topology. There is a subbasis consisting of the following sets, for compact $K \subseteq \QQdisc$ and open $U \subseteq \RR$
	\begin{align*}
		V(K, U) = \{ \alpha \in \RR | \alpha K \subseteq U\}
	\end{align*}
	Clearly, $V(\{1\}, (a,b)) = (a,b)$. Open intervals generate the Euclidean topology on $\RR$, so the compact-open topology is no coarser than the natural one.

	Only finite sets of $\QQdisc$ are compact. We write $V(K,U)$ as an intersection of finitely many open sets.
	\begin{align*}
		V(K,U) = \bigcap _{k \in K} V(\{k\}, U)
	\end{align*}
	Also, $U$ is open so it can be written as $U = \bigcup_{i \in I} (a_i, b_i)$ and clearly $V(K, U) = \bigcup_{i \in I} V(K, (a_i, b_i))$. Finally, notice that $V(\{k\}, (a_i, b_i))) = (a_i/k, b_i/k)$. Hence,
	\begin{align*}
		V(K,U) = \bigcap_{k \in K} \bigcup _{i \in I} (a_i/k, b_i/k)
	\end{align*}
	which is clearly open under the natural topology of $\RR$ (we omit the case $k=0$ because $V(\{0\}, U) = \emptyset$ or $\RR$ so it follows trivially). This shows that the compact-open topology gives $R$ its natural topology.

	Hence,
	\begin{align*}
		\condext^0 (\condRRdisc, \condRR) = \condRR^{2^{\aleph_0}}
	\end{align*}
	The next enriched extension group is not known by the author. By our methods described, one would use a two-term free resolution
	\begin{align*}
		0 \to \ZZ[K] \to \ZZ[\QQ] \to \QQdisc \to 0
	\end{align*}
	to obtain a distinguished triangle. The main problem is the behaviour of the map $\ZZ[K] \to \ZZ[\QQ]$, since it is difficult to obtain an explicit description of the kernel of $\ZZ[\QQ] \to \QQ$.

	It is conjectured by the author that $\condext^1 (\condQQdisc, \condRR) = 0$, by comparison with the category of abelian groups. Additionally, due to the fully faithful embedding from the category of LCA groups \cite[Corollary 4.9]{condensed}, we know that $\condext^1 (\condQQdisc, \condRR) (*) = 0$ by the known result there.
\end{example}

\section{A prismatic construction of the real numbers}
\label{section-prismatic}

This section constructs the real numbers as a quotient of condensed rings. The motivation for such a construction lies in the theory of $p$-liquid vector spaces, a full subcategory of condensed vector spaces stable under all limits, colimits and extensions \cite{analytic}. This is the functional analysis analogue to the solid theory for condensed abelian groups \cite{condensed}.

The building blocks for this theory are condensed objects of the form $\mathcal{M}_p (S)$, which can be thought of as $p$-measures on a profinite set $S$. To obtain a category similar to solid abelian groups, we would want these to be projective. This is not quite the case, but one can obtain the close result that $\ext^i (\mathcal{M}_p (S) , V) = 0$ for $i > 0$ and any $p'$-Banach vector space with $0 < p < p' \leq 1$.

To simplify the problem, Scholze and Clausen build $\RR$ as a quotient of discrete spaces. This section will only explore this quotient, without any further analysis of the liquid theory.

\subsection{The main result}

We first start with the condensed ring from which we will construct the real numbers.

\begin{defi}
	For $0<r<1$ and $c > 0$ we define the condensed set $\condZZTrc$ with $S$-valued points
	\begin{align*}
		\condZZTrc(S) := \{\sum_{n \gg -\infty} a_n T^n | a_n \in C(S, \ZZ), \forall s \in S \; \sum |a_n (s)| r^n \leq c\}
	\end{align*}
	This allows the definition of the condensed ring
	\begin{align*}
		\ZZTr = \bigcup_{c > 0} \ZZTrc
	\end{align*}
\end{defi}

\begin{rem}
	One would need to check that $\condZZTrc$ is a condensed set. Rather than carrying out a construction with Kan extensions, we will prove in the next subsection that it is the condensation of a T1 topological space.

	Also, observe that one obtains a ring at each section after the union, so the resulting condensed set $\ZZTr$ is actually a condensed ring.
\end{rem}

We now have the following result.

\begin{theorem}
	\label{thm-main}
	Let $0 < r' < r < 1$ and consider the map
	\begin{align*}
		\theta_{r'}:\ZZTr & \to \condRR \\
		\sum a_n T^n & \mapsto \sum a_n (r')^n
	\end{align*}
	This is a surjection and the kernel at every section is a principal ideal generated by a non-zero divisor.
\end{theorem}

\begin{rem}
	This result looks initially rather unintuive but it is best understood in terms of decimal expansion. Indeed, set $r' = \frac{1}{10}$. Then, this is equivalent to writing numbers in the form
	\begin{align*}
		\sum_{n \gg - \infty} a_n 10^{-n}
	\end{align*}
	which is precisely decimal expansion. In this case, we will prove later that the kernel is generated at every section by the polynomial $10T - 1$, which would be the expected behaviour.
\end{rem}

We highlight the similarity with the following number-theoretic result.

\begin{theorem}
	Let $K$ be a complete discretely valued extension of $\QQ_p$ with perfect residue field $k$. Then, the ring of integers $\mathcal{O}_K$ can be written as a quotient of $W(k)[[T]]$ by a principal ideal generated by an Eisenstein polynomial $f$.
\end{theorem}

This is the reason for calling this a prismatic construction. By Scholze's and Bhatt's recent work in prismatic cohomology, the pair $(W(k)[[T]], (f))$ is a prism \cite{prismatic}.

\subsection{A helpful ring}

We first look at a non-condensed version of $\ZZTrc$.

\begin{defi}
	For $0 < r < 1$ and $c > 0$ define the topological space
	\begin{align*}
		\ZZTrc := \{ \sum_{n \gg - \infty} a_n T^n | a_n \in \ZZ, \sum |a_n|r^n \leq c \}
	\end{align*}
	With the topology induced by the $T$-adic norm
	\begin{align*}
		\norm{\sum a_n T^n} = \delta^{v_T(f)}
	\end{align*}
	for some fixed $\delta \in (0,1)$ where
	\begin{align*}
		v_T (f) = \sup\{n \in \ZZ: a_n = 0\} 
	\end{align*}
\end{defi}

The following lemma shows that our notation is consistent.

\begin{lemma}
	\label{lemma-condZZTrc}
	The condensed set $\condZZTrc$ is the condensation of $\ZZTrc$.
\end{lemma}

\begin{proof}
	Fix extremally disconnected $S$. We must prove that
	\begin{align*}
		C(S,\ZZTrc) = \condZZTrc(S) = \{ \sum_{n \gg - \infty} a_n T^n | a_n \in C(S, \ZZ), \forall s \in S\; \sum |a_n(s)| r^n  \leq c \}
	\end{align*}
	Let $F \in C(S, \ZZTrc)$. Then, for any $s \in S$ we can write
	\begin{align*}
		F(s) = \sum_{n \in \ZZ} a_n(s) T^n
	\end{align*}
	for some $a_n : S \to \ZZ$ not necessarily continuous, where for all $s \in S$, $\sum |a_n(s)|r^n \leq c$. The space $S$ is compact, so its image is bounded. Hence, there is $k \in \ZZ$ such that $a_n = 0$ for all $n < k$. It is clear that for all $s \in S$, $\sum |a_n| r^n \leq c$.

	We must now prove the $a_n$ are continuous. We will show that for any $A \subseteq S$, $a_n(\overline{A}) \subseteq \overline{a_n(A)}$. Since $\ZZ$ has the discrete topology, this amounts to proving $a_n(\overline{A}) \subseteq a_n(A)$.

	Let $s \in \overline{A}$. Then, $F$ is continuous so $F(s) \in \overline{F(A)}$ and for any $\epsilon > 0$, there exists $s_\epsilon \in A$ such that $\norm{F(s) - F(s_\epsilon)}< \epsilon$. Setting $\epsilon < \delta^n$ we see that the coefficients of $T^k$ must coincide for $k \leq n$, so $a_n(s) = a_n (s_\epsilon) \in a_n (A)$. Hence, $a_n$ is continuous.

	To prove the other direction, one must show that if $a_n \in C(S, \ZZ)$ and for all $s \in S$, $\sum_{n \gg - \infty} |a_n(s)| r^n \leq c$, then
	\begin{align*}
		F = \sum_{n \gg - \infty} a_n T^n \in C(S, \ZZTrc)
	\end{align*}
	Let $A \subseteq S$ and $s \in \overline{A}$. Each $a_n$ partitions $S$ into finitely many clopen subsets, so by taking intersections, for any $N \in \NN$ there are clopen sets $U_1,\dots , U_k$ such that $a_n$ is constant on $U_i$ for $n \leq N$.

	$s \in U_i$ for some $i$. Since $U_i$ is open and $s \in \overline{A}$, $U_i \cap A \neq \emptyset$. Let $s' \in U_i \cap A$. $a_n$ is constant on the $U_i$ for $n \leq N$ so
	\begin{align*}
		\norm{F(s) - F(s')} \leq \delta^N \to 0
	\end{align*}
	as $N \to \infty$ so $F(s)$ is indeed a limit point.
\end{proof}

While $\condZZTr$ is not the condensation of a compact space, it is built by compact sets in a similar way to the real numbers.

\begin{lemma}
	\label{lemma-ZZTrccompact}
	$\ZZTrc$ is compact.
\end{lemma}

\begin{proof}
	We will prove the stronger result that it is a profinite, which is equivalent to being a limit of finite discrete spaces \cite[\href{https://stacks.math.columbia.edu/tag/08ZY}{Tag 08ZY}]{stacks-project}.

	Firstly, $\sum |a_i| r^{i+k} = r^k \sum |a_i| r^{i+k}$ so multiplying by $T^k$ is an isomorphism $\ZZTrc \to \ZZ((T))_{r, r^k c}$. Hence, multiplying by a high enough power of $T$, we may assume that $c < 1$ (remember that $r < 1$).
	
	Then,
	\begin{align*}
		\ZZTrc \subseteq \ZZ[[T]] = \lim_{\longleftarrow} \prod_{n=0}^m \ZZ \cdot T^n
	\end{align*}
	where the inverse limit is of discrete topological spaces, and the right-hand space has the $T$-adic topology.

	Define
	\begin{align*}
		\ZZTrc^m = \{ \sum_{n = 0}^m a_n T^n | a_n \in \ZZ, \sum_{n =0}^m |a_n| r^n \leq c\} \subseteq \prod_{n=0}^m \ZZ \cdot T^n
	\end{align*}
	Since $\sum |a_n| r^n \leq c$ we have finite choices of $a_n$ so the sets $\ZZTrc^m$ are finite.

	In an analogous way to $\ZZ[[T]]$ we can write $\ZZTrc$ as an inverse limit of the $\ZZTrc^m$ and it inherits the $T$-adic topology.
\end{proof}

To understand the condensed ring $\ZZTr$ we first study the following subring of power series.

\begin{defi}
	Let $0 < r < 1$. Define the ring
	\begin{align*}
		\ZZ((T))_{> r} = \{\sum_{n \gg -\infty} a_n T^n | a_n \in \ZZ, \; \exists r' > r : |a_n|(r')^n \to 0\}
	\end{align*}
\end{defi}

\begin{rem}
	The condition $|a_n|(r')^n \to 0$ tells us the following. $|a_n|(r')^n$ must be bounded as a sequence so for any $r < s < r'$,
	\begin{align*}
		\norm{\sum_{n \gg - \infty} a_n s^n} \leq \sum_{n \gg - \infty} (|a_n| (r')^n) \left( \frac{s}{r'} \right)^n < \infty
	\end{align*}
	by comparison with the geometric series, so the radius of convergence is $\geq r'$. Hence, $\ZZ((T))_{>r}$ are all holomorphic functions on some neighbourhood on $B(0,r) \setminus \{0\}$.

	Notice that naturally $\ZZ((T))_{>r} \subseteq \condZZTr(*)$.
\end{rem}

This ring is extensively studied in \cite{analytic} and \cite{arithmetic}. We quote the following result.

\begin{theorem}
	\label{thm-zztrpid}
	Let $0 < r < 1$. Then, $\ZZ((T))_{>r}$ is a principal ideal domain. The non-zero prime ideals are the following.
	\begin{enumerate}
		\item For any $x \in \CC^*$ with $|x| \leq r$, the kernel of the map
			\begin{align*}
				\ZZ((T))_{>r} \to \CC : \sum a_n T^n \mapsto \sum a_n x^n
			\end{align*}
			generated by some $f \in 1 + T \ZZ[[T]]$ whose only zeroes on $\overline{B}(0,r)$ are at $x$ and $\overline{x}$ with multiplicity $1$. This map is surjective onto $\RR$ when $x \in \RR$.
		\item For any prime number $p$, the ideal $(p)$.
		\item For any prime number $p$ and any topologically nilpotent unit $x \in \overline{\QQ_p}$, the kernel of the map
			\begin{align*}
				\ZZ((T)) \to \overline{\QQ_p}: \sum a_n T^n \mapsto a_n x^n
			\end{align*}
	\end{enumerate}
\end{theorem}

\subsection{Proof of the main result}

We will need a few lemmas in order to conclude the proof. The main proof strategy will be restricting ourselves to the compact sets $\ZZTrc$ and using the fact that extremally disconnected spaces are projective in the category of compact Hausdorff spaces.

\begin{lemma}
	\label{lemma-thetacontinuous}
	Let $0 < r < 1$ and let $r' < r$. Then, the map
	\begin{align*}
		F: \ZZTrc \to \RR, \; \sum_{n \gg - \infty} a_n T^n \mapsto \sum_{n \gg - \infty} a_n (r')^n
	\end{align*}
	is continuous.
\end{lemma}

\begin{proof}
	For $f = \sum_{n \gg - \infty} a_n T^n \in \ZZTrc$ we must have $a_n \leq c \cdot r^{-n}$ for all $n$. If $\norm{f-g} < \delta^N$, the first $N$ coefficients match so
	\begin{align*}
		\norm{F(f) - F(g)} \leq 2 \sum_{n > N} c r^{-n} (r')^n  = 2 c \frac{\left(\frac{r'}{r}\right)^N}{1 - \frac{r'}{r}} \to 0
	\end{align*}
	as $N \to \infty$ so the function is continuous.
\end{proof}

\begin{lemma}
	\label{lemma-ZZTrccover}
	Let $K \subseteq \RR$ be compact. Then, there is some $c > 0$ such that
	\begin{align*}
		K \subseteq \theta_{r'} (\ZZTrc)
	\end{align*}
	where $\theta_{r'} : \ZZTrc \to \RR$ is the map in \cref{thm-main} evaluated at the point.
\end{lemma}

\begin{proof}
	Without loss of generality, assume $K = [-M, M]$ since any compact subset of $\RR$ is contained in one such set.

	Fix $x \in K\setminus\{0\}$, and pick $n$ minimal so that $(r')^n \leq |x|$. This minimum is well-defined since $(r')^n \to 0$ as $n \to \infty$ and $(r')^n \to \infty$ as $n \to - \infty$. Also, there is a lower bound on $n$ that is uniform for all $x \in K$ (precisely, $n \geq \log M/\log r'$).

	Pick an integer $a_n \in \ZZ$ such that
	\begin{align*}
		\left| \frac{x}{(r')^n} - a_n \right| < 1
	\end{align*}
	and let $x' = x - a_n (r')^n$, so by our choice of $a_n$, $|x'| < (r')^n \leq |x|$ and $x' \in K$.

	Due to the choice of minimal $n$,
	\begin{align*}
		(r')^n \leq |x|  < (r')^{n-1}
	\end{align*}
	So we must have
	\begin{align*}
		|a_n| = \left| \frac{x}{(r')^n} - \frac{x'}{(r')^n}\right| \leq \frac{|x|}{(r')^n} + \frac{|x'|}{(r')^n} < \frac{1}{r'} + 1
	\end{align*}
	We now repeat the process with $x'$, choosing an $n'$ minimal so that $(r')^{n'} \leq |x'|$. By our choice of $x'$ we must have $n' > n$, and by the same reasoning $|a_{n'}| < 1 + \frac{1}{r'}$. Hence,
	\begin{align*}
		x = \sum_{n \gg - \infty} a_n (r')^n 
	\end{align*}
	where $|a_n| < 1 + 1/r'$ for all $n$. Then, $x = \theta_{r'} (\sum a_n T^n)$. Finally,
	\begin{align*}
		\sum |a_n| r^n \leq (1 + \frac{1}{r'}) \sum_{n \geq \log M /\log r'} r^n
	\end{align*}
	which is a bound only dependent on $M$, so $\theta_{r'}(\ZZTrc) \supseteq K$ for some $c > 0$.
\end{proof}

\begin{lemma}
	\label{lemma-kerneltheta}
	The kernel of $\theta_{r'}$ as in \cref{thm-main} is generated by some $f_{r'} \in \ZZ((T))_{> r}$ at every section.
\end{lemma}

\begin{proof}
	Consider the map
	\begin{align*}
		\ZZ((T))_{> r} \to \RR: \sum a_n T^n \mapsto \sum a_n (r')^n
	\end{align*}
	By \cref{thm-zztrpid}, the kernel is generated by some $f_{r'} \in \ZZ((T))_{> r}$, where $f \in 1 + T\ZZ((T))$ and only vanishes at $r'$ within $\overline{B}(0,r)$.

	We claim that the kernel of $\theta_{r'}$ is generated by $f_{r'}$ at every section (since it has constant functions on the coefficients, $f_{r'} \in \ZZTr (S)$ for every $S$).

	Let $S$ be extremally disconnected, and take $g = \sum_{n \gg -\infty} a_n T^n \in \ZZ((T))_r (S)$ such that $\theta_{r'} (g) = 0$. Now write $f_{r'} = 1 + \sum_{n \geq 1} b_n T^n$ and notice that near $T = 0$ we can write
	\begin{align*}
		\frac{1}{f_{r'}} = \frac{1}{1 + \sum_{n \geq 1} b_n T^n} = 1 + \sum_{k=0}^\infty (-1)^k \left(\sum_{n \geq 1} b_n T^n \right)^k \in \ZZ[[T]]
	\end{align*}
	This is well-defined in a neighbourhood of $\overline{B}(0,r)$ except at $T = r'$, where we have a simple pole. Also, all the coefficients of the series are (constant) continuous functions $S \to \ZZ$.

	Then, $g \cdot f^{-1}_{r'}$ has integer coefficients and a removable singularity at $T = r'$. Write $|g|_{r,s}$ for $\sum |a_n (s)| r^n$ so that $|g|_{r,s} \leq c$ for some $c > 0$ and all $s \in S$. Note that $f \in \ZZ((T))_{>r}$ so $f^{-1}_{r'}$ is defined in an annulus centered at $0$ that contains $r$, and by absolute convergence $|f^{-1}_{r'}|_{r,s} \leq c'$ for some $c' > 0$ and all $s \in S$ ($f_{r'}$ is constant as a function on $S$). Hence, $|g \cdot f^{-1}_{r'}|_{r,s} \leq |g|_{r,s} \cdot |f^{-1}_{r'}|_{r,s} \leq  c \cdot c'$ for all $s \in S$. Therefore, $g \cdot f^{-1}_{r'} \in \ZZTr (S)$.

	It follows that $g \in (f_{r'})$.

\end{proof}

\begin{lemma}
	\label{lemma-thetaepic}
	The function $\theta_{r'}: \ZZTr \to \condRR$ is an epimorphism.
\end{lemma}

\begin{proof}
	Evaluating at every section, we must check that there is a surjection
	\begin{align*}
		\ZZTr (S) \to C(S, \RR)
	\end{align*}
	Let $f \in C(S, \RR)$. Then, $f(S)$ is compact so by \cref{lemma-ZZTrccover} there is some $c > 0$ such that
	\begin{align*}
		\theta_{r'} (\ZZTrc) \supseteq f(S)
	\end{align*}
	$f(S) \subseteq \RR$ is compact, so it is closed. By \cref{lemma-ZZTrccompact}, $\ZZTrc$ is compact and by \cref{lemma-thetacontinuous}, the map $\theta_{r'}$ is continuous when restricted to $\ZZTrc$. Hence, $\theta_{r'}^{-1}(f(S)) \cap \ZZTrc$ is a closed subset of a compact space, so it is compact. 
	\[
		\begin{tikzcd}
			& S \arrow{d}{f} \arrow[ld, dashed]	\\
			\theta_{r'}^{-1}(f(S))\cap \ZZTrc \arrow[r, twoheadrightarrow, "\theta_{r'}"] & f(S)
		\end{tikzcd}
	\]
	We have the following solid diagram of compact Hausdorff spaces. Since extremally disconnected sets are projective in the category of compact Hausdorff spaces, this factors through some continuous map $g: S \to \theta_{r'}^{-1}(f(S)) \cap \ZZTrc \subseteq \ZZTrc$.

	By \cref{lemma-condZZTrc}, $g \in \condZZTrc(S) \subseteq \ZZTr(S)$ so we have $\theta_{r'} (g) = f$. The map is surjective at every section, so it is an epimorphism.

\end{proof}

\begin{proof}
	(of \cref{thm-main}) We will show that
	\begin{align*}
		0 \to \ZZTr \xrightarrow{f_{r'}} \ZZTr \xrightarrow{\theta_{r'}} \condRR \to 0
	\end{align*}
	is a short exact sequence of condensed abelian groups. To prove the first injection, we must show that for profinite $S$, $f_{r'} \in \ZZTr(S)$ is not a zero divisor. Evaluating at $s \in S$, $f_{r'}$ does not change because it only has constant functions as its coefficients.

	Furthermore, $\ZZTr(S)(s)$ is an integral domain, so if $f(s) g(s) = 0$ then $g(s) = 0$. Then, if $fg = 0$, $g(s) = 0$ for all $S$ and $g = 0$.

	By \cref{lemma-kerneltheta}, $\ker (\theta_{r'}) = f_{r'} \cdot \ZZTr$. Finally, by \cref{lemma-thetaepic}, $\theta_{r'}$ is an epimorphism so this is a short exact sequence.
\end{proof}

\begin{example}
	We now return to the mentioned example, setting $r' = \frac{1}{10}$. Since $f_{r'} \in 1 + T\ZZ[[T]]$, we have $\deg f_{r'} \geq 1$. Evaluating at the point,
	\begin{align*}
		\ZZTr(*) = \{\sum_{n\gg - \infty} a_n T^n | \sum a_n r^n < \infty\}
	\end{align*}
	Within this ring, we must have $f_{r'} | 10T - 1$ since it is in the kernel. Write $f_{r'} g = 10T - 1$. Since $f_{r'}$ does not vanish at $0$ and $10T-1$ does not have a pole at $0$, $g \in \ZZ[[T]]$. But then for degree reasons it is clear $\deg g = 0$ so $f_{r'} = \pm (10T -1)$.
\end{example}

\newpage
\bibliographystyle{amsalpha}
\bibliography{references}

\end{document}